\documentclass[10pt]{article}

\usepackage[margin=1in]{geometry} 
\usepackage{pxfonts} 

\usepackage[english]{babel} 
\usepackage{amsmath, amsfonts, amsthm, amssymb} 
\usepackage[scr=boondoxo,cal=euler]{mathalfa} 
\usepackage{tikz-cd} 
\usepackage{indentfirst} 
\usepackage{fancyhdr} 
\usepackage{graphicx} 
\usepackage{enumitem} 
\usepackage{microtype} 
\usepackage{pdfpages} 
\usepackage{centernot} 
\usepackage{ragged2e} 
\usepackage{pinlabel}
\usepackage{caption}
\usepackage[hidelinks]{hyperref} 

\allowdisplaybreaks 
\usetikzlibrary{decorations.pathmorphing}

\theoremstyle{plain}
\newtheorem{thm}{Theorem}
\newtheorem*{thm*}{Theorem}
\newtheorem{lem}[thm]{Lemma}
\newtheorem*{lem*}{Lemma}
\newtheorem{cor}[thm]{Corollary}
\newtheorem*{cor*}{Corollary}

\newtheorem*{prop*}{Proposition}

\newtheorem*{ques*}{Question}

\theoremstyle{definition}

\newtheorem*{df*}{Definition}
\newtheorem*{dfs*}{Definitions}

\newtheorem*{exercise*}{Exercise}

\theoremstyle{remark}

\newtheorem*{rem*}{Remark}

\usepackage{etoolbox}
\patchcmd{\thmhead}{(#3)}{#3}{}{}
\makeatletter
\g@addto@macro\bfseries{\boldmath}
\makeatother

\newcommand{\Z}{\mathbf{Z}} 
\newcommand{\Q}{\mathbf{Q}} 
\newcommand{\F}{\mathbf{F}}

\newcommand{\rBN}{\smash{\overline{\mathrm{BN}}}}

\newcommand{\rKh}{\smash{\overline{\Kh}}}

\DeclareMathOperator{\rank}{rank}
\DeclareMathOperator{\rk}{rk}

\DeclareMathOperator{\Kh}{Kh}

\title{A minimality property for knots without Khovanov 2-torsion}
\author{Onkar Singh Gujral and Joshua Wang}
\date{}
\begin{document}
\maketitle
\begin{abstract}
	A conjecture of Shumakovitch states that every nontrivial knot has 2-torsion in its Khovanov homology. We show that if a knot $K$ has no $2$-torsion in its Khovanov homology, then the rank of its reduced Khovanov homology is minimal among all knots obtainable from $K$ by a proper rational tangle replacement. It follows, for example, that unknotting number 1 knots have $2$-torsion in their Khovanov homology.
\end{abstract}

\vspace{15pt}

Shumakovitch conjectured that every nontrivial knot has an element of order $2$ in its Khovanov homology \cite[Conjecture 1]{MR3205577}. The conjecture has been verified for some infinite families of knots (see for example \cite{MR3205577,MR2147419,MR3205574}) and has withstood large computational searches. In this note, we provide topological evidence for the conjecture, and we verify the conjecture for a large class of knots that include all unknotting number 1 knots.

Two links differ by a \textit{rational tangle replacement} if they agree outside of a ball, and if within the ball, each is a rational tangle. A rational tangle replacement is \textit{proper} if the arcs of the two rational tangles connect the same end points \cite{iltgen2021khovanov,MR4550372}. Changing a crossing is an example of a proper rational tangle replacement, while resolving a crossing is an example of a non-proper rational tangle replacement. In the following statement, $\Kh(K)$ and $\rKh(K)$ denote the unreduced and reduced Khovanov homology groups of $K$, respectively, thought of as abelian groups with bigradings suppressed.

\begin{thm}\label{thm:mainthm}
	Suppose $K$ is a knot such that there is no $2$-torsion in $\Kh(K)$. If $J$ is a knot that differs from $K$ by a proper rational tangle replacement, then \[
		\rank \rKh(K) \leq \rank \rKh(J).
	\]
\end{thm}

\begin{cor}\label{cor:unknottingNumber1}
	Any knot whose unknotting number is 1 has 2-torsion in its Khovanov homology. 
	More generally, if $K$ is a nontrivial knot that can be obtained from the unknot or a trefoil by a proper rational tangle replacement, then $\Kh(K)$ contains $2$-torsion. 
\end{cor}
\begin{proof}[Proof of Corollary~\ref{cor:unknottingNumber1}]
	Let $J$ be the unknot or a trefoil, and let $K$ be obtained from $J$ by a proper rational tangle replacement. If there is no $2$-torsion in $\Kh(K)$, then $\rank \rKh(K) \leq \rank \rKh(J) \leq 3$ by Theorem~\ref{thm:mainthm}. The rank of $\rKh(K)$ cannot be $3$ since then $K$ would be a trefoil \cite{Baldwin_2022}, which has $2$-torsion in its Khovanov homology. Since the rank of $\rKh(K)$ is odd, it must be 1 so $K$ is the unknot \cite{MR2805599}.
\end{proof}

Our proof of Theorem~\ref{thm:mainthm} combines the main result of Iltgen--Lewark--Marino \cite{iltgen2021khovanov} with an observation of Kotelskiy--Watson--Zibrowius \cite[Proposition 9.3]{kotelskiy2019immersed} using the following lemma. 

\begin{lem}\label{lem:mainlemma}
Let $\F$ be a field, and suppose $M$ and $N$ are finitely-generated modules over the polynomial ring $\F[X]$ of the form \[
	M = (\F[X])^r \oplus \bigoplus_{i=1}^m \frac{\F[X]}{X^{a_i}} \qquad N = (\F[X])^s \oplus \bigoplus_{i=1}^n \frac{\F[X]}{X^{b_i}}
\]where $r,m,s,n \ge 0$ and $a_1,\ldots,a_m,b_1,\ldots,b_n \ge 1$. Furthermore, suppose $f\colon M \to N$ and $g\colon N \to M$ are $\F[X]$-module maps for which $f\circ g = X$ and $g \circ f = X$. If the numbers $a_1,\ldots,a_m$ are all at least $2$, then $m \leq n$.
\end{lem}
\begin{proof}
Let $X_M$ and $X_N$ denote the structural maps $X\colon M \to M$ and $X\colon N \to N$, respectively. Our aim is to establish the inequality $m = \dim_\F \ker X_M \leq \dim_\F \ker X_N = n$.
Let $C\coloneqq g^{-1}(\ker X_M)$, and observe that \[
    C \supseteq \ker X_N \supseteq X_NC.
\]The first inclusion is straightforward to verify. For the second inclusion, suppose $y \in C$ and observe that $X_N(X_N y) = f(g(X_N y)) = f(X_M g(y)) = 0$. We claim that $g|_C\colon C \to \ker X_M$ is surjective. Since the numbers $a_1,\ldots,a_m$ are all at least two, any element $y$ in the kernel of $X_M$ lies in the image of $X_M$, and therefore may be written as $y = X_Mz = g(f(z))$, which proves the claim. Next, note that $g$ sends $X_NC$ to zero so $g|_C$ gives a surjection $C/X_NC \to \ker X_M$. Thus \[
    \dim_\F \ker X_M \leq \dim_\F C - \dim_\F X_NC = \dim_\F \frac{C}{\ker X_N} + \dim_\F \ker X_N - \dim_\F X_NC = \dim_\F \ker X_N
\]where the last equality follows from the isomorphism $C/\ker X_N \to X_NC$ induced by $X_N$. 
\end{proof}

\begin{proof}[Proof of Theorem~\ref{thm:mainthm}]
Let $\rBN(K)$ denote the reduced Bar-Natan homology of $K$ with rational coefficients. It is a rank $1$ finitely-generated graded module over $\Q[H]$ where $H$ has nonzero degree, so we may write \[
    \rBN(K) \cong \Q[H]\oplus \bigoplus_{i=1}^m \frac{\Q[H]}{H^{a_i}} \qquad \rBN(J) \cong \Q[H] \oplus \bigoplus_{i=1}^n \frac{\Q[H]}{H^{b_i}}
\]where $a_1,\ldots,a_m,b_1,\ldots,b_n$ are positive. By hypothesis, there is no 2-torsion in $\Kh(K)$, so \cite[Proposition 9.3]{kotelskiy2019immersed} implies that the numbers $a_1,\ldots,a_m$ are all at least $2$. Furthermore, \cite[Proof of Proposition 9.3]{kotelskiy2019immersed} also gives $\rk \rKh(K) = 1 + 2m$ and $\rk \rKh(J) = 1 + 2n$.

By \cite[Proof of Theorem 1.1]{iltgen2021khovanov}, there are $\F[H]$-module maps $f\colon \rBN(K) \to \rBN(J)$ and $g\colon \rBN(J) \to \rBN(K)$ satisfying $f\circ g = H$ and $g \circ f = H$. We note that the complex $\llbracket D\rrbracket$ over $\Z[G]$ associated to a diagram $D$ considered in \cite{iltgen2021khovanov} recovers the reduced Bar-Natan complex as $\llbracket D \rrbracket \otimes_{\Z[G]} \Q[H]$ where $\Z[G] \to \Q[H]$ sends $G \mapsto -H$. By Lemma~\ref{lem:mainlemma}, we obtain \[
	\rk \rKh(K) = 1 + 2m \leq 1 + 2n = \rk \rKh(J).\qedhere
\]
\end{proof}

\theoremstyle{definition}
\newtheorem*{ack}{Acknowledgments}
\begin{ack}
	We thank Akram Alishahi, John Baldwin, Artem Kotelskiy, Lukas Lewark, Tom Mrowka, and Raphael Zentner for interesting discussions. OSG thanks his advisor Lisa Piccirillo for useful conversations, support, and advice, and JW thanks Peter Kronheimer for his continued guidance, support, and encouragement. OSG and JW were partially supported by the Simons Collaboration on New Structures in Low-Dimensional Topology, and JW was also partially supported by the NSF MSPRF grant DMS-2303401.
\end{ack}

\bibliographystyle{amsalpha}
\bibliography{2torsion}

\vspace{10pt}

\textit{Department of Mathematics, MIT}

\textit{Email:} \texttt{onkar@mit.edu}

\vspace{10pt}

\textit{Department of Mathematics, MIT}

\textit{Email:} \texttt{joshuaxw@mit.edu}

\end{document}